\documentclass[12pt]{amsart}

\usepackage{mathpazo}
\usepackage{verbatim}
\usepackage{multirow}

\usepackage{amsmath,amsthm,amssymb,array,enumerate,parskip,tabularx}
\usepackage{graphicx}

\usepackage{xcolor} %improved colors
\usepackage{mathrsfs}

\usepackage[pagebackref]{hyperref}
\definecolor{dark-red}{RGB}{161,0,155}
\definecolor{dark-blue}{rgb}{0.15,0.15,0.6}
\definecolor{dark-green}{rgb}{0.15,0.6,0.15}
\hypersetup{
    colorlinks, linkcolor=dark-red,
    citecolor=dark-blue, urlcolor=dark-green
}

\renewcommand*{\backref}[1]{}
\renewcommand*{\backrefalt}[4]{%
  \ifcase #1 %
    \relax
  \or
    $\uparrow$#2.%
  \else
    $\uparrow$#2.%
  \fi%
}

\usepackage{mathtools}
\usepackage{colonequals}
\usepackage{fullpage}

\newcommand{\Z}{\mathbb{Z}} %integers
\newcommand \rH {{\rm H}}

\numberwithin{equation}{section}

\newcommand \PP {{\mathbb P}^1}
\newcommand \ZZ {{\mathbb Z}}
\newcommand \CC {{\mathbb C}}
\newcommand \QQ {{\mathbb Q}}

\newcommand{\PSL}{{\rm PSL}}
\newcommand{\frakH}{\mathfrak{H}}
\newtheorem{lem}{Lemma}[section]
\newtheorem{lemma}[lem]{Lemma}
\newtheorem{theorem}[lem]{Theorem}

\newtheorem{corollary}[lem]{Corollary}
\newtheorem{example}[lem]{Example}
\newtheorem{proposition}[lem]{Proposition} 
\newtheorem{intheorem}{Theorem}

\theoremstyle{definition}
\newtheorem{remark}[lem]{Remark}

\newtheorem{goal}[lem]{Goal} 

\title[The classifying element for quotients of Fermat curves]{The classifying element for quotients of Fermat curves}
\author{Juanita Duque-Rosero}
\address{Boston University}
\email{juanita@bu.edu}
\author{Rachel Pries}
\address{Colorado State University}
\email{pries@colostate.edu}

\begin{document}

\begin{abstract}
Suppose $C$ is a cyclic Galois cover of the projective line branched at the three points $0$, $1$, and $\infty$. 
Under a mild condition on the ramification, we determine the structure of the graded Lie algebra 
of the lower central series of the fundamental group of $C$ in terms of a basis which is well-suited to studying 
the action of the absolute Galois group of $\QQ$.



%SAVE secondary
%11F06  	Structure of modular groups and generalizations; arithmetic groups 
%11F67  	Special values of automorphic $L$-series, periods of automorphic forms, cohomology, modular symbols
%11G30 Arithmetic geometry, Curves of arbitrary genus or genus $\ne 1$ over global fields
%13A50  Commutative ring theory; Actions of groups on commutative rings; invariant theory 
%14F20  	Cohomology theory; tale and other Grothendieck topologies and (co)homologies

\end{abstract}

\maketitle

\section{Introduction}

In this companion paper to \cite{DPWgreen},
the main goal is to determine information about the \'etale fundamental group of a cyclic Belyi curve using its lower central series.  To explain the meaning of this, we first provide some background.

Suppose $X$ is a smooth projective curve of genus $g$ defined over a number field $K$.
Let $\bar{K}$ be the algebraic closure of $K$ and suppose that $X_{\bar{K}}$ is connected.
Let $\pi = [\pi]_1 = \pi_1(X)$ be the \'etale fundamental group  
and let $\rH_1(X)$ be the \'etale homology group of $X_{\bar{K}}$.

For $m \geq 2$, let $[\pi]_m$ be the $m$th subgroup of the lower central series $\pi = [\pi]_1 \supset [\pi]_2 \supset [\pi]_3 \supset \cdots$;
specifically, $[\pi]_m=\overline{[\pi, [\pi]_{m-1}]}$ is the closure of the subgroup generated by 
commutators of elements of $\pi$ with elements of $[\pi]_{m-1}$. 
Then $\rH_1(X) \cong [\pi]_1/[\pi]_2$.

Consider the graded Lie algebra ${\rm gr}(\pi) = \oplus_{m \geq 1} [\pi]_m/[\pi]_{m+1}$ 
of the lower central series for $\pi$, \cite{lazard, serre65}.
Let $F$ be the free profinite group on $2g$ generators and consider its graded Lie algebra 
${\rm gr}(F)= \oplus_{m \geq 1} {\rm gr}_m(F)$. 
By \cite[Theorem, page 17]{labute70}, 
there is an element $\rho$ of weight $2$ such that 
\[{\rm gr}(\pi)  \cong {\rm gr}(F)/\overline{\langle \rho \rangle}.\] 

Thus ${\rm gr}(\pi)$ is determined by the subgroup $\langle \rho \rangle$.
By \cite[Corollary 8.3]{hain},
\[\displaystyle [\pi]_2/[\pi]_3 \cong \left(\rH_1(X) \wedge \rH_1(X)\right)/{\rm Im}(\mathscr{C}),\]
where 
\begin{equation}\label{defmapC}
{\mathscr{C}}\colon \rH_2(X) \to \rH_1(X) \wedge \rH_1(X)
\end{equation}
is the dual map to the cup product map $\rH^1(X) \wedge \rH^1(X) \to \rH^2(X)$.
Since $\rH_2(X)\cong\Z(1)$, the image $\mathrm{Im}(\mathscr{C})$ is cyclic. 
A generator $\Delta$ for $\mathrm{Im}(\mathscr{C})$ is called a {\it classifying element}.

There is a formula for $\Delta$ in terms of a set of generators of the 
fundamental group of $X$ that satisfies certain properties; see \eqref{Edelta} for details.  
In the context of \'etale homology groups, we would like additional information,
specifically the following.

\begin{goal} \label{G1}
Find a formula for a classifying element using a basis for $\rH_1(X)$ which is well-suited 
for studying the action of the absolute Galois group $G_K$ and the action of ${\rm Aut}(X)$. 
\end{goal}

In \cite{DPWgreen}, the authors realized this goal when $X=X_n\colon x^n+y^n=z^n$ is the Fermat curve of degree $n$, 
for any integer $n \geq 3$.
When $n=p$ is an odd prime satisfying Vandiver's conjecture, 
and $K=\QQ(\zeta_p)$ is the cyclotomic field, then the information about the action of $G_K$ on $\rH_1(X_p,\ZZ/p\ZZ)$
comes from \cite{anderson} and \cite{Baction}.

In this paper, we realize Goal~\ref{G1} when $X=W_{n,k}$ is a cyclic Belyi curve, 
namely a curve with affine equation 
\begin{equation} \label{EcyclicB}
v^n=u(1-u)^k,
\end{equation} for any odd integer $n$ and integer $k$ with $1 \leq k \leq n-2$.
This curve admits a Galois $\mu_n$-cover $\phi\colon W_{n,k} \to {\mathbb P}^1$ branched at $3$ points $0$, $1$,  and $\infty$. 
We restrict to the case that the cover is totally ramified at the ramification points $\eta_0$, $\eta_1$, and $\eta_\infty$;
this is true if and only if $\gcd(n,k(k+1))=1$; in particular, it is true for all $1 \leq k \leq n-2$ if $n$ is prime.

The main result of the paper is Theorem~\ref{Tmaintheorem}; writing $W=W_{n,k}$,
we determine a classifying element $\Delta \in \rH_1(W) \wedge \rH_1(W)$ 
for all such pairs $(n,k)$. 
This determines the isomorphism class of ${\rm gr}(\pi)$ as a graded Lie group with the action of 
$\mu_n \subset {\rm Aut}(W)$. 

Here is some notation needed to state the result.
Let $U=W - \eta_\infty$, where $\eta_\infty$ is the unique point not on the affine chart \eqref{EcyclicB}.
In Section~\ref{ScurveDef}, we define some loops $E_1, \ldots, E_{n-1}$ in the fundamental group $\pi_1(U)$ \eqref{Dloops}.
Let $[E_1], \ldots, [E_{n-1}]$ denote the images of $E_1, \ldots, E_{n-1}$ in the homology group $\rH_1(U)$.
In Corollary~\ref{CBasisEi}, we prove that these form a basis for $\rH_1(U)$.
This basis for $\mathrm{gr}_1(F) \cong \rH_1(W)\cong \rH_1(U)$
yields a basis for ${\rm gr}_2(F) \cong \rH_1(W) \wedge \rH_1(W)$ given 
by $\{[E_i] \wedge [E_j] \mid 1 \leq i < j \leq n-1\}$. 
We state our main result in terms of that basis. 

\begin{intheorem}[Theorem~\ref{Tmaintheorem}] \label{Tintro}  
Suppose $1 \leq k \leq n-2$ and $\mathrm{gcd}(n, k(k+1))=1$. 
Let $W$ be the smooth projective curve with affine equation $v^n=u(1-u)^k$.

Let $c$ be the integer such that $1 \leq c \leq n-1$ and $c$ is the multiplicative inverse of $k+1$ modulo $n$.
Then a classifying element $\Delta$ for $W$ is given by 
\[\Delta= \sum_{1 \leq I < J \leq n-1} c_{I,J} [E_I] \wedge [E_J],\]
where 
\[c_{I,J} = \begin{cases}
- 1 & \text{ if } J-I \equiv j(k+1) - 1 \bmod{n},\text{ or}\\
+ 1  & \text{ if } J-I  \equiv j(k+1) \bmod{n}, 
\end{cases}\]
for some $j$ such that $1 \leq j \leq c-1$; and $c_{I,J}=0$ otherwise.
\end{intheorem}

For the proof, we first rely on Lemma~\ref{Lleeway}, which states that a formula for $\Delta$ can be found by 
expressing a loop $L$ around $\eta_\infty$ as a product of commutators of elements in $\pi_1(U)$.
Then the main ingredients of the proof are the topology and Galois theory of branched coverings.
In Section~\ref{S3}, we use these to find a combinatorial formula for $L$; the formula is first described using 
edges that generate the fundamental groupoid of $U$ with respect to $\{\eta_0, \eta_1\}$, 
and then re-expressed in terms of the loops $E_1, \ldots, E_{n-1}$.

In Section~\ref{Sexamples}, we provide examples for arbitrary odd $n$ and certain values of $k$.

\begin{remark} \label{Rtdbu}
The motivation for studying covers of $\mathbb{P}^1$ 
branched at three points originates with Grothendieck's 
program for understanding the absolute Galois group of $\QQ$.
The cyclic Belyi curves play an important role in the study of Galois theory, Hurwitz spaces, 
and abelian varieties with complex multiplication. 
Since $W=W_{n,k}$ is a quotient of the Fermat curve $X_n$, 
it might be possible to prove Theorem~\ref{Tintro} with a top-down approach, using the result of \cite{DPWgreen}. 
In working on this problem, we realized that there are many advantages with the direct approach, see Remark~\ref{RRtdbu}.
\end{remark}

\begin{remark} \label{Rhomologycoefficients}
Theorem~\ref{Tintro} is valid both for the homology with coefficients in $\ZZ$
and for the \'etale homology which has coefficients in a finite or $\ell$-adic ring.
In Section~\ref{Smodular}, the proof
relies on the modular symbols of Manin \cite{Manin} (and a result of Ejder \cite{E:integral}).
However, we follow an approach which is compatible with the results in \cite{anderson} 
and \cite{Baction} about the \'etale homology with coefficients in $\ZZ/n\ZZ$ and 
the action of the absolute Galois group upon it, because this will be important in future applications.

After choosing an embedding $\QQ \subset \mathbb{C}$ and applying Riemann's Existence Theorem, 
we may identify the profinite completion of $\rH_1(U(\mathbb{C}))$ with the \'etale homology $\rH_1(U)$.
Similarly, we may identify the profinite completion of $\pi_1(U(\mathbb{C}))$ with the
\'etale fundamental group $\pi_1(U)$.
Thus we can consider the elements $E_1, \ldots, E_{n-1}$ to be in the topological fundamental group 
or in the \'etale fundamental group; 
similarly, we can consider the elements $[E_1], \ldots, [E_{n-1}]$ to be in the simplicial homology or in the 
\'etale homology.
A similar comparison holds for other objects in the paper.
\end{remark}

\subsection*{Acknowledgments}
We would like to thank the anonymous referee for their thoughtful and constructive feedback.  Duque-Rosero was supported by a grant from the Simons Foundation (550029, to Voight).  Pries was supported by NSF grant DMS-22-00418.

\section{The fundamental group of cyclic Belyi curves}\label{ScurveDef}

We describe the geometry of cyclic Belyi curves and their relationship to Fermat curves.
We state some facts about the fundamental group and the classifying element $\Delta$.

Let $n \geq 3$ be a positive integer.
Let $\zeta = e^{2 \pi i/n}$ be a fixed primitive $n$th root of unity.
Fix an integer $k$, with $1 \leq k \leq n-2$.
For simplicity, we assume throughout the paper that $\mathrm{gcd}(n, k(k+1))=1$; 
this assumption is true if $n$ is prime, which is the situation of future applications of this paper.

\subsection{The geometry of cyclic Belyi curves}

Let $W=W_{n,k}$ be the smooth projective curve having affine equation:
\begin{equation} \label{EcyclicBelyi}
v^n = u(1-u)^k.
\end{equation}

Let $\eta_0$ be the point $(u,v)=(0,0)$; let $\eta_1$ be the point $(u,v)=(1,0)$.  
The hypothesis that $\mathrm{gcd}(n,k+1)=1$ implies that there is a unique point $\eta_\infty$ on $W$ 
which is not on this affine chart.
Consider the open affine subset $U = W - \eta_\infty$.

There is a $\mu_n$-Galois cover $\phi\colon W \to \mathbb{P}^1$, given by $\phi(u,v) \mapsto u$.
The Galois group is generated by the automorphism $\epsilon \in \mathrm{Aut}(W)$ of order $n$ that acts by 
$\epsilon((u,v))=(u, \zeta v)$.
The cover $\phi$ is totally ramified at the points $\eta_0$, $\eta_1$, and $\eta_\infty$, which lie over the 
branch points $u=0,1,$ and $\infty$ respectively.

By the Riemann-Hurwitz formula, the genus of $W$ is $g=(n-1)/2$.

Any $\mu_n$-cover of ${\mathbb P}^1$ branched at three points, which is totally ramified at one point, 
admits an equation of the form \eqref{EcyclicBelyi} for some $k$ with $1 \leq k \leq n-2$.  
The condition of being totally ramified over the other two branch points is equivalent to
$\mathrm{gcd}(n, k(k+1))=1$.

\subsection{The fundamental group} 
Throughout the paper, composition of paths and loops is denoted by the symbol $\cdot$ and written from left to right. 
Note that $U \subset W$ is a real surface of genus $g=(n-1)/2$ with $1$ puncture.
We choose the base point $\eta_1$.
There exist loops $a_i, b_i$ for $1 \leq i \leq g$ and $c_\infty$, with base point $\eta_1$, such that 
$\pi_1(U)$ has a presentation 
\begin{equation}\label{pi1UWpresentation}
\pi_1(U) = \langle a_i, b_i, c_\infty \mid i=1,\ldots,g \rangle/\prod_{i=1}^g [a_i,b_i] c_\infty.\end{equation}

Without loss of generality, we choose the loop $c_\infty$ to circle the puncture $\eta_\infty$
and to have no set-theoretic intersection 
with the loops $a_i, b_i$ for $1 \leq i \leq g$. 
This can be arranged using a standard gluing of a $1$-punctured
polygon with $4g$ sides, with the sides labeled consecutively by 
$a_1, b_1, a_1^{-1}, b_1^{-1}, \ldots, a_g, b_g, a_g^{-1}, b_g^{-1}$.

The homology group $\rH_1(U)$ is equipped with an intersection pairing, defined using Poincar\'e duality 
and the cup product on compactly supported cohomology.
Let $\bar{a}_i, \bar{b}_i, \bar{c}_\infty$ denote the images of $a_i, b_i, c_\infty$ 
in $\rH_1(U)$.  Note that $\bar{c}_\infty$ is trivial. 

Without loss of generality,
we can suppose that the images of $\bar{a}_i, \bar{b}_i$ in $\rH_1(U)$ form a standard symplectic basis.
Since $U$ has only one puncture, there is an isomorphism $\rH_1(U) \cong \rH_1(W)$.
These two facts imply that a generator of ${\rm Im}({\mathscr C})$ as in \eqref{defmapC} can be identified with:
\begin{equation} \label{Edelta}
\Delta_W = \sum_{i=1}^{g} \bar{a}_i \wedge \bar{b}_i \in \rH_1(U) \wedge \rH_1(U).
\end{equation}

\subsection{The second graded quotient in the lower central series}

By \eqref{Edelta}, $\Delta_W= \sum_{i=1}^g \bar{a}_i \wedge \bar{b}_i$.
We would like to determine $\Delta_W$ in terms of a basis of $\rH_1(U) \wedge \rH_1(U)$ for which 
we have information about the action of the absolute Galois group.
To do this, we investigate the element $T \colonequals  \prod_{i=1}^g [a_i,b_i] = c_\infty^{-1}$ in $\pi_1(U)$.

We need the following two results.
The first shows that $\Delta_W$ does not depend on the representation
as a product of commutators.

\begin{lemma} \label{Lleeway} \cite[Lemma~2.2]{DPWgreen}
Suppose $r_1, \ldots, r_N, s_1 \ldots, s_N$ are loops in $U$, 
with images $\bar{r}_i, \bar{s}_i$ in $\rH_1(U)$.
If
\[T \ {\rm is \ homotopic \ to \ } [r_1, s_1] \cdot \cdots \cdot [r_N, s_N],\] 
then $\sum_{i=1}^g \bar{a}_i \wedge \bar{b}_i = \sum_{i=1}^N \bar{r}_i \wedge \bar{s}_i$ in 
$\rH_1(U) \wedge \rH_1(U)$.
\end{lemma}

The next lemma will help simplify later calculations.  

\begin{lemma} \label{Lsimplify} \cite[Lemma~2.3]{DPWgreen}
Suppose $\alpha, \beta, \gamma \in \pi_1(U)$. 
\begin{enumerate}
\item If $\alpha \gamma \in [\pi_1(U)]_2$, then $\gamma \alpha \in [\pi_1(U)]_2$, and
$\alpha \gamma$ and $\gamma \alpha$ have the 
same image in the quotient $[\pi_1(U)]_2/[\pi_1(U)]_3$.

\item If $\gamma^{-1} \alpha \gamma \beta \in [\pi_1(U)]_2$, then $\alpha \beta \in [\pi_1(U)]_2$, and
the difference between the images of $\gamma^{-1} \alpha \gamma \beta$ and $\alpha \beta$ in 
$[\pi_1(U)]_2/[\pi_1(U)]_3$
is $\gamma \wedge (-\alpha)$.
\end{enumerate}
\end{lemma}

\subsection{The fundamental groupoid}

More generally, we consider the fundamental groupoid $\pi_1(U, \{\eta_0, \eta_1\})$ of $U$ 
with respect to the points $\eta_0$ and $\eta_1$. 
Let $\beta_W$ be the path in $U$, 
which begins at the base point $\eta_0$ and ends at $\eta_1$, 
given by 
\begin{equation} \label{Ebeta}
\beta_W =\left(t, \sqrt[n]{t(1-t)^k}\right) \ {\rm for} \ t \in [0,1].
\end{equation}
Here the symbol $\sqrt[n]{t(1-t)^k}$ denotes the real-valued and positive $n$th root.

Recall that $\epsilon((u,v)) = (u, \zeta v)$.
For $0 \leq i \leq n-1$, we define a path in $U$, which begins at $\eta_0$ and ends at $\eta_1$ by
\begin{equation} \label{Ebetaij}
e_{i} = \epsilon^i \beta_W.
\end{equation}

Define $\tau = \beta_W^{-1}$, where the 
inverse of a path is the path traversed in the opposite direction. 
Note that $e_i^{-1} = \epsilon^i \beta_W^{-1}$.

We define some loops in $U$ with base point $\eta_1$:
for $0 \leq i \leq n-1$, let 
\begin{equation} \label{Dloops}
E_i \colonequals e_i^{-1} \cdot e_0 = \epsilon^i\tau \cdot \tau^{-1}.
\end{equation}

The loop $E_i$ implicitly depends on $k$.
Note that
\begin{equation} \label{Elittleebige}
E_i \cdot E_j^{-1} = e_i^{-1} \cdot e_0 \cdot e_0^{-1} \cdot e_j= e_{i}^{-1} \cdot e_j.
\end{equation}

If $i=0$, then $E_i$ is trivial in $\pi_1(U, \{\eta_0, \eta_1\})$.
Corollary~\ref{CBasisEi} shows that the converse is true.  

\subsection{Comparison with the Fermat curve}

It is well-known that $W=W_{n,k}$ is a quotient of the Fermat curve $X_n\colon x^n+y^n=z^n$ of degree $n$ 
(see, e.g.\ \cite[Chapter~8]{Murty}). 
Let $Z_F$ be the set of $n$ points where $z=0$ on $X_n$.
The open affine subset $U_F=X_n - Z_F$ is given by the set of points 
$(x,y)$ such that $x^n+y^n=1$.  

In \cite[(2.g)]{DPWgreen}, the authors defined a path $\beta$ in $U_F$.
We remark that $\beta_W$ is the image of $\beta$ under the map $U_F \to U$ that takes $(x,y)$ to $(x^n,xy^k)$. 

The automorphism group of $X_n$ contains two automorphisms $\epsilon_0$ and $\epsilon_1$ of order $n$ that 
commute; these act by $\epsilon_0((x,y)) = (\zeta x, y)$ and $\epsilon_1((x,y)) = (x, \zeta y)$.
Let $H=H_{n,k}$ be the subgroup of $\mathrm{Aut}(X_n)$ generated by $h=\epsilon_0 \epsilon_1^{- k^{-1} \bmod n}$.

\begin{lemma} \label{Lquotient}
The cyclic Belyi curve $W_{n, k}$ is the quotient of the Fermat curve $X_n$ of degree $n$ by $H_{n,k}$.
The fiber of $X_n$ over $\eta_\infty$ is the set of $n$ points in $Z_F$; the fiber of $X_n$ over $\eta_0$ (resp.\ $\eta_1$)
is the set of $n$ points on $U_F$ where $x=0$ (resp.\ $y=0$).
\end{lemma}

\begin{proof}
There is a well-defined inclusion from the function field of $W_{n,k}$ to the function field of $X_n$ 
determined by $u \mapsto x^n$ and $v \mapsto xy^k$.
This inclusion has degree $n$.  The first claim follows since $u$ and $v$ are fixed by $h$.
The claims about the fibers follow by calculation.
\end{proof}

\begin{remark} \label{RRtdbu}
Here are the reasons that proving Theorem~\ref{Tmaintheorem} with a top-down approach would be 
more complicated. 

First, the combinatorial description of the loop in Section~\ref{S3} is substantially 
easier than the description of $n$ loops in \cite{DPWgreen}. 
This is because the cover $\phi\colon W \to {\mathbb P}^1$ has degree $n$, rather than 
$n^2 = \mathrm{deg}(X_n \to \mathbb{P}^1)$, and also because
there is one point $\eta_\infty$ of $W$ above $\infty$ rather than the $n$ points of $Z_F$.
The theoretical description of the boundary of a simple closed disk containing $\eta_\infty$ 
is easier than that for the boundary of a simple closed disk in $X_n$ containing the $n$ points of $Z_F$.

Second, in Section~\ref{Smodular}, the homology group $\rH_1(W)$ is a subspace of a free module of rank one
over $\ZZ[\mu_n]$, while $\rH_1(X_n)$ is a subquotient of a free module of rank one 
over the more complicated group ring $\ZZ[\mu_n \times \mu_n]$.  
The formula for $\Delta_W$ in \eqref{Edelta} is easier than for the Fermat curve where 
there is a non-trivial homomorphism $\wedge^2 \rH_1(U_F) \to \wedge^2 \rH_1(X_n)$. 

Third, the next lemma shows that the chosen basis elements of $\rH_1(X_n)$ 
have slightly complicated images in terms of the chosen basis elements of $\rH_1(W)$. 
For these reasons, the formula in Theorem~\ref{Tmaintheorem} is easier to 
write down and prove using a direct approach.
\end{remark}

Recall from \cite[Section~2.3]{DPWgreen}, the definitions of the paths $\{e_{i,j}\}_{0 \leq i,j \leq n-1}$ in $U_F$ and the loops
\[E_{i,j} = e_{0,0} \cdot  (e_{0,j})^{-1} \cdot e_{i,j} \cdot (e_{i,0})^{-1}.\]
By \cite[Lemma~4.1]{DPWgreen}, the images of $\{[E_{i,j}]\}_{1\le i,j\le n-1}$ form a basis for $\rH_1(U_F)$.

\begin{lemma} Under the map $\pi_1(U_F)\to\pi_1(U)$ induced by the map $U_F\to U$:
\begin{enumerate}
\item The image of $e_{i,j}$ in $\pi_1(U, \{\eta_0, \eta_1\})$ is $e_{i+jk}$.
\item The image of $E_{i,j}$ in $\pi_1(U)$ is $e_0 \cdot e_{jk}^{-1} \cdot e_{i+jk} \cdot e_i^{-1}$, which starts and ends at $\eta_0$.
\item The adapted loop $e_{jk}^{-1} \cdot e_{i+jk} \cdot e_i^{-1} \cdot e_0$, which starts and ends at $\eta_1$, equals 
$E_{jk} \cdot E_{i+jk}^{-1} \cdot E_i$.
\end{enumerate}
\end{lemma}

\begin{proof}
This follows from the equalities $u=x^n$ and $v=xy^k$ in the proof of Lemma~\ref{Lquotient}.
\end{proof}

\section{The classifying element of the Belyi curve} \label{S3}

We continue to study the curve $W=W_{n,k}$ with affine equation $v^n=u(1-u)^k$.
Recall that $T = \prod_{i=1}^g [a_i, b_i]=c_\infty^{-1}$ is in the homotopy class of the boundary of a disk in $W$ 
that contains the point $\eta_\infty$.
In Proposition~\ref{Pwhichloop}, we find a loop homotopic to $T$ written 
in terms of the elements $E_{i}$ in $\pi_1(U)$.
We then analyze the ordering of the loops $E_{i}$ in $T$ combinatorially. 
This enables us to find an explicit formula for $\Delta \in \rH_1(U) \wedge \rH_1(U)$, 
using a basis on which we have some information about the action of the absolute Galois group,
see Theorem~\ref{Tmaintheorem}.

\subsection{Gluing sheets of an unramified cover}

Recall that $\phi \colon W \to {\mathbb P}^1$ is the $\mu_n$-Galois cover given by $(u,v) \mapsto u$.
The cover $\phi$ is branched at $\{u=0, 1, \infty\}$ and is ramified at 
$\{\eta_0=(0,0), \eta_1=(1,0), \eta_\infty\}$. 
In this section, we remove some paths in $W$ and ${\mathbb P}^1$ in order to work with an unramified cover.

Given the equation $v^n=u(1-u)^k$, 
the inertia type of $\phi$ is the $3$-tuple $(1, k, -(1+k))$.
This means that the canonical generators of inertia at $\eta_0, \eta_1,$ and $\eta_\infty$ 
are $\zeta^1, \zeta^k,$ and $\zeta^{-1-k}$, respectively.
In other words, the chosen generator $\epsilon$ of the Galois group of $\phi$ acts on a uniformizer at 
each ramification point by this root of unity, respectively.

We make a slit cut along the positive real line
in $\mathbb{P}^1(\CC)$ from $u=1$ to $u=0$ and another from $u=1$ to $u=\infty$.
We choose a base point $\underline{u}_1$ close to $u=1$ and in the lower half plane; a technical term for
this is a tangential base point at $u=1$.  We also make a short slit cut $\underline{t}$ from $u=\underline{u}_1$ to $u=1$.

Let $P^\circ$ be the complement of these three slit cuts in $\mathbb{P}^1(\CC)$.
Let $W^\circ$ be the complement of the $3n$ paths in $W$ that lie above these three slit cuts.

\begin{lemma}
The restriction $\phi\colon W^\circ \to P^\circ$ is unramified. 
\end{lemma}

\begin{proof}
The monodromy around $u=0,1,\infty$ is multiplication by 
$\zeta^1, \zeta^k,$ and $\zeta^{-(k+1)}$, respectively.  
So a loop going around all 3 of these points is multiplication by $1$.
Therefore, the monodromy action of $\pi_1(P^\circ)$ on $W^\circ$ is trivial, 
which proves the claim.  
\end{proof}

Thus $W^\circ$ is a disjoint union of $n$ connected components, which are called sheets.
We need to label the regions of $\PP$ near the slit cuts and the edges along the boundary of $W^\circ$.  
It might be helpful to look at Figure~\ref{F113} for reference.

For the regions of $\PP$: let $N$ denote the region of the upper half plane which is close to the positive real axis;
let $E$ denote the region of the lower half plane which is close to the values $u \geq 1$ on the real axis; and
let $S$ denote the region of the lower half plane which is close to the values $0 \leq u \leq 1$ on the real axis.

For the edges along the boundary of $W^\circ$, we start by labeling the unique 
edge $\tau$ of $W^\circ$ having the following property:
it is on the path $\beta_W$; and it is on the right hand-side as one travels from $\eta_1$ to $\eta_0$, 
meaning that it lies above $N$ rather than $S$.  

Let $0 \leq i < n$.  The action of $\epsilon^i$ on $W^\circ$ allow us to label $n$ of the edges as $\epsilon^i \tau$; 
these are the edges that lie above the interval $[0,1]$ in $N$.
Furthermore, we label by $R_i$ the sheet of $W^\circ$ that contains $\epsilon^i \tau$.
In each sheet $R_i$: we label by $\epsilon^i \alpha$ the unique edge that lies above the ray 
$[1,\infty)$ in $N$; and we label by $\epsilon^i \xi$ the unique edge that lies above $\underline{t}$ in $S$.
This completes the labeling of one side of each of the $3n$ slit cuts in $W^\circ$.

The inertia type gives the information needed to glue the sheets together along the 
slit cuts to obtain the ramified cover $\phi$ of Riemann surfaces.

\begin{lemma} For $0 \leq i < n$:
\begin{enumerate}
\item The edge $\epsilon^i \tau$ on $R_i$ glues with the unique edge in $R_{i - 1 \bmod n}$ 
that lies above the interval $[0,1]$ in $S$. 
\item The edge $\epsilon^i \alpha$ on $R_i$ glues with the unique edge in $R_{i - (k+1) \bmod n}$ 
that lies above the interval $[1,\infty)$ in $E$. 
\item The edge $\epsilon^{i} \xi$ on $R_{i}$ glues with
the unique edge in $R_{i}$ that lies above $\underline{t}$ in $E$.
\end{enumerate}
\end{lemma}

\begin{proof}
\begin{enumerate}
\item Imagine a simple closed loop around $\eta_0$ that crosses the edge $\epsilon^i \tau$; 
as one travels around this loop in a counterclockwise direction, the fact that the generator of inertia at $\eta_0$ is 
$\zeta^1$ implies that one passes from the sheet $R_{i - 1 \bmod n}$ to the sheet $R_i$.
Thus the edge $\epsilon^i \tau$ on $R_i$ must be glued with the unique edge in $R_{i - 1 \bmod n}$ 
that lies above the interval $[0,1]$ in $S$. 

\item Imagine a simple closed loop around $\eta_\infty$ that crosses the edge $\epsilon^i \alpha$; 
as one travels around this loop in a counterclockwise direction, the fact that the generator of inertia at $\eta_\infty$ is 
$\zeta^{-(k+1)}$ implies that one passes from the sheet $R_i$ to the sheet $R_{i - (k+1) \bmod n}$.
Thus the edge $\epsilon^i \alpha$ on $R_i$ must be glued with the unique edge in $R_{i - (k+1) \bmod n}$ 
that lies above the interval $[1,\infty)$ in $E$. 

\item
Imagine a simple closed loop around a lift of $\underline{t}$ that crosses the edge $\epsilon^i \xi$; 
as one travels around this loop in a counterclockwise direction, one should stay on the same sheet 
since $\underline{t}$ is not a branch point. 
Thus the edge $\epsilon^{i} \xi$ on $R_{i}$ must be glued to the unique edge in $R_i$
that lies above the interval $\underbar{t}$ in $E$. 
\end{enumerate}
\end{proof}

\subsection{Lifting of a loop}

Let $\mathbb{H}^+$ be the upper half plane and $\mathbb{H}^-$ be the lower half plane.
Let $\tilde{u}_1$ be the lift of the base point $\underline{u}_1$ which is on $\xi$ in $R_0$.
In this section, all loops in $\mathbb{P}^1(\CC)$ (resp.\ $W(\CC)$) have 
base point $\underline{u}_1$ (resp.\ $\tilde{u}_1$).

In $\mathbb{P}^1(\CC)$, consider a counterclockwise simple closed loop $Z_\circ$ around $\infty$.
It is homotopic to a clockwise simple closed loop $Z$ 
that first crosses from $\mathbb{H}^-$ to $\mathbb{H}^+$ at some point in ${\mathbb R}^{u <0}$ 
and then crosses from $\mathbb{H}^+$ to $\mathbb{H}^-$ at some point in ${\mathbb R}^{u>1}$.  
Without loss of generality, we can suppose that this last crossing occurs at an arbitrarily large value of $u$.

Let $\tilde{Z}_\circ$ be a lift of $Z_\circ$ to $W$.  
Let $\tilde{Z}$ be a lift of $Z$ to $W$. 

Our goal now is to describe the loop $\tilde{Z}$ in terms of the edges 
$\epsilon^i \tau$, and then in terms of the loops $E_i$, for $0 \leq i < n$.
Recall that $e_i = \epsilon^i \tau$ and $E_i=\epsilon^i\tau\cdot\tau^{-1}$. 
For $0 \leq j \leq n-1$, let 
\begin{equation} \label{Eloopstart}
L_j \colonequals \epsilon^{j(n-k-1)+1}\tau\cdot (\epsilon^{j(n-k-1)}\tau)^{-1} = E_{j(n-k-1)+1} \cdot E_{j(n-k-1)}^{-1},
\end{equation}
where the second equality follows from \eqref{Elittleebige}.
Let \[L\colonequals L_0 \cdot L_1 \cdot \cdots\cdot L_{n-1}.\]  
We view $L$ as a word in $\{E_i, E_i^{-1}\}_{0 \leq i <n}$, including the elements $E_0, E_0^{-1}$ 
as placeholders even though they are trivial in homology. 

\begin{proposition} \label{Pwhichloop}
The loop $c_\infty$ is homotopic to $L$.
\end{proposition}

\begin{proof}
The loop $c_\infty$ is homotopic to $\tilde{Z}$, because $\tilde{Z}$ is homotopic to $\tilde{Z}_\circ$, which is a counterclockwise simple closed loop around $\eta_\infty$.
So it suffices to prove the formula for $\tilde{Z}$.

The loop $\tilde{Z}$ in $W$ covers the loop $Z$ in $\mathbb{P}^1$ exactly $n$ times.  This is because it is homotopic 
to the simple closed loop $\tilde{Z}_\circ$ around the point $\eta_\infty$, which is a ramification point for $\phi$ with 
ramification degree $n$.  Each revolution begins above $S$, then goes above $N$, then switches sheets and goes above $E$.

The loop $\tilde{Z}$ starts on the sheet $R_0$.
Then it crosses the edge $\alpha$ onto the sheet $R_{n-(k+1)}$, where it swings by the point 
$\epsilon_{n-(k+1)} \tilde{u}_1$.  That completes one of the $n$ revolutions.  
The next revolution is similar but starts on the sheet $R_{n-(k+1)}$.  
So the subindex on each path increases by $n-(k+1)$ after each revolution.

The first revolution is homotopic to a path that traces in a clockwise direction along the outside of the slits, 
going near the following points, in this order:
\[\tilde{u}_1 \mapsto \eta_1 \mapsto \eta_0 \mapsto \eta_1 \mapsto \eta_\infty \mapsto \eta_1 \mapsto 
\epsilon^{n-(k+1)} \tilde{u}_1.\]
This path is the composition of the following paths, written from left to right, 
with the first four on the sheet $R_0$ and the last two on 
the sheet $R_{n-(k+1)}$:
\[\xi^{-1} \cdot \epsilon \tau \cdot \tau^{-1} \cdot \alpha \cdot \alpha^{-1} \cdot \epsilon^{n-(k+1)} \xi.\] 

The paths $\alpha$ and $\alpha^{-1}$ cancel.  
Also the last path $\epsilon^{n-(k+1)} \xi$ on the first revolution cancels with the first path on the next revolution,
and the first path $\xi^{-1}$ on the first revolution cancels with the last path on the last revolution;
so these paths can be ignored, leaving only $L_0$. 

Thus $\tilde{Z}$ is homotopic to $L_0 \cdot \epsilon^{n-(k+1)} L_0 \cdot \cdots \cdot \epsilon^{(n-1)(n-(k+1))} L_0$.
By  \eqref{Eloopstart}, the equation for $L_{j+1}$ is $\epsilon^{n-(k+1)} L_j$.  
Thus $\tilde{Z}$ is homotopic to $L_0 \cdot L_1 \cdot \cdots \cdot L_{n-1}$. 
\end{proof}

\subsection{Combinatorial analysis of loop} \label{Sformulaforloop}

We need a combinatorial analysis of the ordering of the edges in the loop $L$.  We say that a loop $E_j$ (or $E_j^{-1}$) is between $E_i^{-1}$ and $E_i$ if it is written between them after cyclically permuting the loops so that $E_i^{-1}$ is the leftmost loop in $L$.

Recall that $\mathrm{gcd}(n,k+1)=1$.
Let $c\in\{1,\ldots, n-1\}$ be such that $c\equiv (k+1)^{-1} \bmod n$.  
Note that $c(n-k-1)\equiv n-1\bmod n$. 

\begin{proposition} \label{Pbetween}
Let $0 \leq i < n$. The loops between $E_i^{-1}$ and $E_i$ in $L$ are 
\[\left\{E_{i+ j(n-k-1) + 1}, E_{i+ j(n-k-1)}^{-1}\right\}_{1 \leq j \leq c-1}.\]
\end{proposition}

\begin{proof}
It suffices to prove the result when $i=0$ by symmetry.
When $i=0$, the claim is that the loops in $L$ between $E_0^{-1}$ and $E_0$ in $L$ are
$E_{j(n-k-1) + 1}$ and $E_{j(n-k-1)}^{-1}$ for $1 \leq j \leq c-1$.

To see this, consider the ordering of the loops in $L$:
%SAVE: $j$ is the index of which part of the loop and $j(p-(k+1))$ is the sheet on which that part of the loop lies.  
%The value $c$ is the number of sheets needed to traverse to move from $\tau^{-1}$ to $\tau$.

sheet 0: $L_0 = \epsilon\tau\cdot\tau^{-1}=E_1 \cdot E_0^{-1}$;

sheet $n-k-1$: $L_1 = \epsilon^{(n-k-1)+1}\tau\cdot\left(\epsilon^{n-k-1}\tau\right)^{-1}
=E_{(n-k-1)+1} \cdot E_{n-k-1}^{-1}$;

sheet $2(n-k-1)$: $L_2=\epsilon^{2(n-k-1)+1}\tau\cdot\left(\epsilon^{2(n-k-1)}\tau\right)^{-1}
=E_{2(n-k-1)+1} \cdot E_{2(n-k-1)}^{-1}$;

and continuing on to sheet $c(n-k-1)$: $L_c = \tau\cdot\left(\epsilon^{n-1}\tau\right)^{-1}= E_0 \cdot E_{n-1}^{-1}$.
%SAVE last sheet $(p-1)(p-k-1)$: $\epsilon^{(n-1)(n-k-1)+1} \tau \cdot \left(\epsilon^{(n-1)(n-k-1)} \tau\right)^{-1}
%E_{(n-1)(n-k-1)+1} \cdot E_{(n-1)(n-k-1)}^{-1}$.

The loop $E_0$ occurs first on the sheet $c(n-k-1)$.
The result follows because the stated loops are the ones that occur in $L_1, \ldots, L_{c-1}$.
\end{proof}

\subsection{Main result}

\begin{theorem} \label{Tmaintheorem}
Let $n$ and $k$ be integers with $1 \leq k \leq n-2$ and $\mathrm{gcd}(n, k(k+1))=1$. 
Let $W$ be the smooth projective curve with affine equation $v^n=u(1-u)^k$.

Let $c$ be the integer such that $1 \leq c \leq n-1$ and $c$ is the multiplicative inverse of $k+1$ modulo $n$.
Then the classifying element $\Delta$ for $W$ is given by 
\[\Delta= \sum_{1 \leq I < J \leq n-1} c_{I,J} [E_I] \wedge [E_J],\]
where 
\[c_{I,J} = \begin{cases}
- 1 & \text{ if } J-I \equiv j(k+1) - 1 \bmod{n}, \text{ or} \\
+ 1  & \text{ if } J-I  \equiv j(k+1) \bmod{n}, 
\end{cases}\]
for some $j$ such that $1 \leq j \leq c-1$; and $c_{I,J}=0$ otherwise.
\end{theorem}

\begin{proof}
By Proposition~\ref{Pwhichloop}, the loop $c_\infty$ is homotopic to $L$.
By Lemma~\ref{Lleeway}, we can determine the image of $L$ in $[\pi_1(U)]_2/[\pi_1(U)]_3$ 
by writing it as a product of commutators.
By applying Lemma~\ref{Lsimplify} repeatedly, this image
is determined by the ordering of the loops in $L$; specifically, if $E'$ lies between $E_i^{-1}$ and $E_i$, then 
$-[E_i] \wedge [E']$ appears in the image.
By Proposition~\ref{Pbetween}, the image contains the terms 
\[-[E_i]\wedge [E_{i+j(n-k-1)+1}]=-[E_i]\wedge [E_{i-j(k+1)+1}]=[E_{i-j(k+1)+1}] \wedge [E_i],\] and 
\[+[E_i]\wedge [E_{i+j(n-k-1)}]=[E_i]\wedge [E_{i-j(k+1)}]= - [E_{i-j(k+1)}] \wedge [E_i],\] for $1 \leq j \leq c-1$. 

Since $T=c_\infty^{-1}$, we negate the coefficients once more; (this is not crucial, since this only scales $\Delta$).  Thus $c_{I,J}=-1$ if $J-I= j(k+1)-1$ and $c_{I,J}=1$ if $J-I= j(k+1)$ for some $1 \leq j \leq c-1$, and 
$c_{I,J}=0$ otherwise.
\end{proof}

\begin{remark} 
The coefficient $c_{I,J}$ is non-zero if exactly one of $E_J$ and $E_J^{-1}$ is between $E_I$ and $E_I^{-1}$ in this loop, indicating that there is a non-trivial commutator involving these elements.  
Specifically, $c_{I,J} =1$ for 
the ordering $E_I,E_J , E_I^{-1} , E_J^{-1}$ and
$c_{I,J} =-1$ for
the ordering $E_I , E_J^{-1} , E_I^{-1} , E_J$. 
\end{remark}

\begin{remark}
If $\mathrm{gcd}(j, n)=1$, then it is possible to compare the classifying elements for the inertia types $(j, jk, -j(k+1))$ and 
$(1,k,-(k+1))$.  Specifically, when we replace $E_i$ by $E_{j i \bmod n}$ in the classifying element for the former, 
we obtain the classifying element of the latter.  For example, $\Delta_{5, (3,1,1)} = -[E_1] \wedge [E_2] - [E_2] \wedge [E_3] - [E_3] \wedge [E_4]$.
Replacing $E_i$ by $E_{3i \bmod 5}$, we obtain the same result as for $\Delta_{5,(1,2,2)}$.

It is also possible to compare the classifying elements after a permutation of the three elements of the inertia type, 
but this involves a more complicated linear transformation on the homology.
\end{remark}

\subsection{Examples when $n=5$} \label{Spictures}

\begin{figure}
\centering
\includegraphics[scale=0.078]{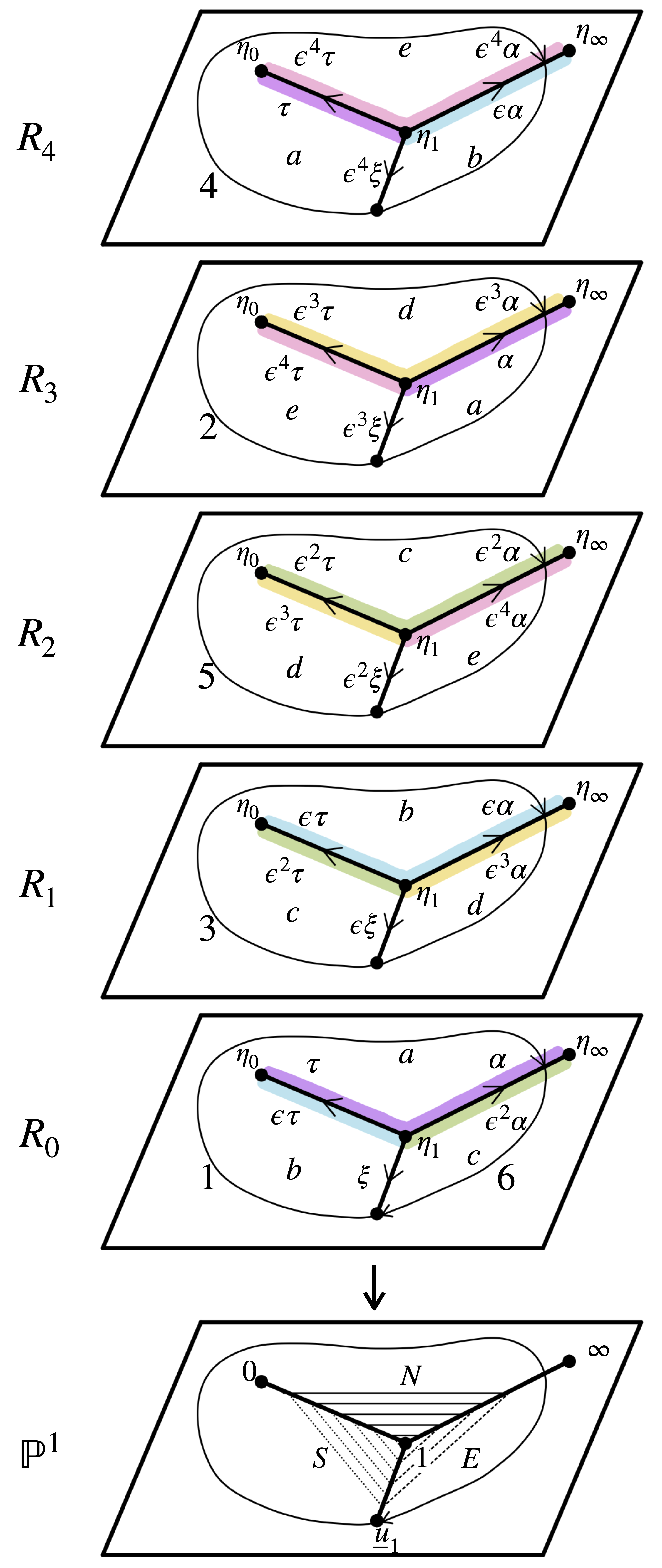} \hspace{0.6cm} \includegraphics[scale=0.078]{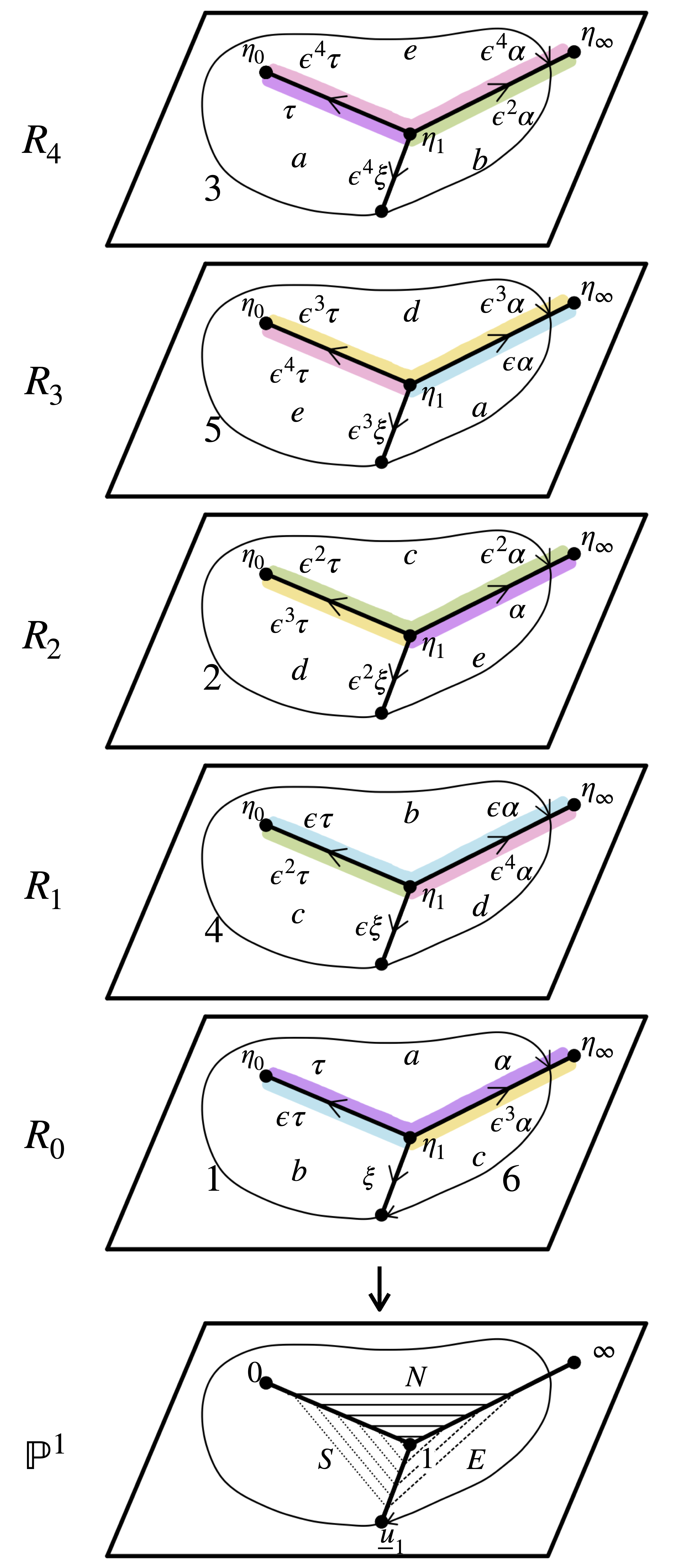}
\caption{Examples: $n=5$ and $k=1$ on the left and $k=2$ on the right.}
\label{F113}
\end{figure}

First, suppose that $k=1$.  
The left-hand side of Figure~\ref{F113} shows this example in detail.  The colors and letters $a$ -- $e$ in the figure represent the gluing of the sheets.

In Section~\ref{Sformulaforloop}, we consider a simple closed loop $L_\circ$ going counterclockwise 
around the point $\infty$ in ${\mathbb P}^1$.  The loop illustrated in the base of the diagram is homotopic to $L_\circ$.
We lift $L_\circ$ to a loop $L$ in $W$. 
The loop $L$ is homotopic to the following composition of paths:
\begin{equation*}\epsilon\tau\cdot\tau^{-1}\cdot\epsilon^4\tau\cdot(\epsilon^3\tau)^{-1}\cdot\epsilon^2\tau\cdot(\epsilon\tau)^{-1}\cdot\tau\cdot(\epsilon^4\tau)^{-1}\cdot\epsilon^3\tau\cdot(\epsilon^2\tau)^{-1}.\end{equation*}

The order of the path is labeled in the figure with steps $1$ -- $6$.
Recall that $E_i = \epsilon^i\tau\cdot\tau^{-1}$ for $0 \leq i \leq 4$.  Note that $E_0$ is trivial.
Then $L$ is homotopic to:
\begin{equation*}
E_1\cdot E_4\cdot E_3^{-1}\cdot E_2\cdot E_1^{-1} \cdot E_4^{-1}\cdot E_3\cdot E_2^{-1}.
\end{equation*}

The ordering of the loops in $L$ 
implies the following formula when $n=5$ and $k=1$:
\begin{equation*}
\Delta=[E_1]\wedge(-[E_2]+[E_3]-[E_4]) + [E_2]\wedge(-[E_3]+[E_4])+[E_3]\wedge (-[E_4]).
\end{equation*}
This formula agrees with Theorem~\ref{Tmaintheorem}.

For comparison, suppose $k=2$.
The right-hand side of Figure~\ref{F113} shows the gluing of the sheets.
In this case, $L$ is homotopic to the following composition of paths:
\begin{equation*}
\epsilon\tau\cdot\tau^{-1}\cdot\epsilon^3\tau\cdot(\epsilon^2\tau)^{-1}\cdot\tau\cdot(\epsilon^4\tau)^{-1}\cdot\epsilon^2\tau\cdot(\epsilon\tau)^{-1}\cdot\epsilon^4\tau\cdot(\epsilon^3\tau)^{-1}.
\end{equation*}

So $L$ is homotopic to:
\begin{equation*}
E_1\cdot E_3\cdot E_2^{-1}\cdot E_4^{-1}\cdot E_2\cdot E_1^{-1}\cdot E_4\cdot E_3^{-1}.
\end{equation*}

From this, we can deduce for $n=5$ and $k=2$ that:
\begin{equation*}
\Delta=E_1\wedge(-[E_3]+[E_4]) + E_2\wedge (-[E_4]).
\end{equation*}

\section{Modular symbols and basis for homology} \label{Smodular}

Fix an integer $n$.  Let $W_k\colonequals W_{n,k}$ denote the smooth projective curve with affine equation $v^n=u(1-u)^k$, 
where $k$ is such that $1 \leq k \leq n-2$ with $\mathrm{gcd}(n, k(k+1))=1$.
We describe the homology group $\rH_1(W_k, \ZZ)$ using Manin's theory of modular symbols from \cite{Manin}, 
following the approach of Ejder \cite{E:integral}.  
In Corollary~\ref{CBasisEi}, we prove that a basis for $\rH_1(W_k, \ZZ)$ as a $\Z$-module is given by $\{[E_i ]\mid 1 \leq i \leq n-1\}$.

\subsection{A modular description of $W_k$}

In $\PSL_2(\ZZ)$, 
consider the congruence subgroup $\Gamma(2)$ and its commutator $\Gamma(2)'$.
Let \begin{equation}\label{EAandB}A\colonequals\begin{bmatrix}1&2\\0&1\end{bmatrix}, 
\mbox{ and }B\colonequals\begin{bmatrix}1&0\\2&1\end{bmatrix}.\end{equation}
Let $\Phi(n)\colonequals\langle A^n, B^n, \Gamma(2)'\rangle$.
Note that $\Phi(n) \subset \Phi_k$, where $\Phi_k$ is  
the congruence subgroup 
\begin{equation}\Phi_k\colonequals\left\langle AB^{-(k^{-1})\bmod n},\,A^n,\,B^n,\,\Gamma(2)'\right\rangle.\end{equation}

Let $\frakH$ denote the upper half plane.
There is an isomorphism between the modular curve $X_{\Phi(n)} \colonequals \bar\frakH/\Phi(n)$
and the Fermat curve $X_n\colon x^n + y^n=z^n$ by \cite[Section~3]{Rohrlich}.  
We now give a similar description of $W_k$.

\begin{lemma}\label{Lindex}
The curve $W_k$ is isomorphic to $X_{\Phi_k}\colonequals\bar\frakH/\Phi_k$.
The index of $\Phi(n)$ in $\Phi_k$ is $n$.
\end{lemma}

\begin{proof}
Possibly after adjusting the isomorphism $X_{\Phi(n)}\cong X_n$,  
we can identify $A$ with the automorphism $\epsilon_0(x,y,z)=(\zeta x,y,z)$ 
and $B$ with the automorphism $\epsilon_1(x,y,z)=(x,\zeta y,z)$, as in \cite[Section 3.3]{E:integral}. 
The first statement is true because $W_k$ is the quotient of $X_n$ by $\langle \epsilon_0\epsilon_1^{-(k^{-1})\bmod n} \rangle$.
The second statement follows since $n=\mathrm{deg}(X_n \to W_k)$.
\end{proof}

Let $\pi\colon \bar{\frakH}\to\bar{\frakH}/\Phi_k$ be the projection map. 

The modular description of $W_k$ allows us to use modular symbols to describe $\rH_1(W_k,\Z)$ as follows.   
A modular symbol is the image of a geodesic from $\alpha$ to $\beta$ in $W_k(\CC)$ for some $\alpha, \beta$ 
in ${\mathbb P}^1(\QQ)$ and it is denoted by $\{\alpha,\beta\}$.
By \cite[Sections~1.3~and~1.5]{Manin}, every $g \in \Phi_k$ determines a modular symbol $[g]=\{\alpha_g,\beta_g\}$ where 
$\alpha_g \colonequals g \cdot 0$ and $\beta_g \colonequals g \cdot i \infty$.
Manin proved \cite[Proposition~1.4 \& Proposition~1.6]{Manin} that the elements of $\rH_1(W_k,\Z)$ are finite sums of the form 
$\sum_m n_m[g_m]$ where $\sum_mn_m(\pi(\beta_{g_m})-\pi(\alpha_{g_m}))=0$ for $n_m\in\Z$.  

We now compute generators for the group of modular symbols of $W_k$.
 
\begin{lemma}\label{Lcosetsphik}
With $A$ and $B$ as in \eqref{EAandB}, the sets $\Phi_kA^rB^s$ and $\Phi_kA^{r+m}B^{s-m(k^{-1})\bmod n}$ are the same right coset of $\Phi_k$ in $\Gamma(2)$ for any $1\le m\le n-1$. In particular, the right coset $\Phi_kA^rB^s$ equals $\Phi_kA^{r+ks}$.
A set of representatives for the cosets of $\Phi_k$ in $\Gamma(2)$ is given by $\Phi_k A^r$ for $0 \leq r \leq n-1$.
\end{lemma}

\begin{proof}
We note that $A,B\in\Gamma(2)$.  Let $\tilde{k}$ be the unique integer such that $k^{-1}\equiv \tilde{k}\bmod n$. For $m=1$, we compute:
\begin{equation*}
\begin{array}{rcl}
\Phi_kA^rB^s &= &\Phi_k(AB^{-\tilde k}) (B^sA^rB^{-s}A^{-r})A^rB^s\\
&=&\Phi_k \left(A^r\left(AB^{s-\tilde k}\right)A^{-r}\left(AB^{s-\tilde k}\right)^{-1} \right)AB^{s-\tilde k}A^{r}\\
&=&\Phi_kA^{r+1}B^{s-\tilde k}.
\end{array}
\end{equation*}
By recursion, we get the desired equality for any $m$.

The second statement follows by taking $m=ks$.

A set of representatives for the right cosets of $\Phi(n)$ in $\Gamma(2)$ is given by $\{A^iB^j\}_{0\le i,j\le n-1}$.  
By Lemma~\ref{Lindex}, the index of $\Phi(n)$ in $\Phi_k$ is $n$.
So there are $n$ right cosets of $\Phi_k$ in $\Gamma(2)$ and each of these 
is the union of $n$ cosets of $\Phi(n)$ in $\Gamma(2)$.
By the first statement,
$\{A^r\}_{0 \leq r \leq n-1}$ is a complete set of representatives. 
\end{proof}

Let $\displaystyle\tau=\begin{bmatrix}0&-1\\1&-1\end{bmatrix}$.
The modular symbol $[\tau]=\{1,0\}$ represents a geodesic that 
starts at $1$ and ends at $0$.
To see this, we compute
\begin{equation} \label{Etauendpoints}
\tau\cdot 0= \begin{bmatrix}0&-1\\1&-1\end{bmatrix}  \begin{bmatrix}0 \\1\end{bmatrix} =  \begin{bmatrix}-1 \\-1 \end{bmatrix} =1,
\mbox{ and } \tau\cdot i \infty =
\begin{bmatrix}0&-1\\1&-1\end{bmatrix}  \begin{bmatrix}1 \\0 \end{bmatrix} =  \begin{bmatrix} 0  \\ 1 \end{bmatrix} =0.
\end{equation}
We use the notation $\tau$ since $[\tau]$ is homotopic to the class of $\tau=\beta_W^{-1}$ as in \eqref{Ebeta}.

\begin{proposition}\label{Pgenerators_modular_group}
The group of modular symbols for $\Phi_k$ is a free $\ZZ$-module of rank $n$ with basis
\begin{equation} \label{Efreen}
\{[A^r\tau] \mid 0 \le r \le n - 1\}.
\end{equation}
\end{proposition}

\begin{proof}
By \cite[Proposition~3.1]{E:integral}, 
the group of modular symbols for the Fermat group $\Phi(n)$ 
is free of rank $n^2+1$ generated by
\[\{[A^iB^j\tau] \mid 1\le i\le n-1,\,0\le j \le n-1\}\cup \{[A^{n-1}B^j] \mid 0 \le j \le n - 1\}\cup\{[B^{n-1}\tau]\}.\]
Since $\Phi(n)\subseteq \Phi_k$, the set of modular symbols for $\Phi_k$ is also generated 
as a $\ZZ$-module by these elements.  All it remains to do is to find the relations between the generators.  By Lemma~\ref{Lcosetsphik}, 
$[A^iB^j]=[A^{i+kj}]$, and so $[A^iB^j\tau]=[A^{i+kj}\tau]$. 
Thus, the modular symbols
\[\{[A^r\tau] \mid 0\le r\le n-1\}\cup \{[A^r] \mid 0 \le r \le n - 1\}\]
are generators for the group of modular symbols of $\Phi_k$.

By \cite[(3.7)]{E:integral} and Lemma~\ref{Lcosetsphik}, for any $1\leq r \leq n-1$, we also have the relation
\begin{equation*}
    [A^r]-[A^{r-1}]=[A^{r-1}\tau]-[A^{r-k}\tau].
\end{equation*}
Taking $r=1$ shows 
$[A]=[\tau]-[A^{1-k}\tau]$.
Working inductively on $i$, we deduce that $[A^i]$ is a $\ZZ$-linear combination of 
the elements in \eqref{Efreen}, completing the proof that this set generates the group of modular symbols.
The properties of being free and rank $n$ follow because this proof uses all of the relations in \cite[Section~3]{E:integral}.
\end{proof}

\begin{proposition} \label{Pbasis}
A basis for the homology group $\rH_1(X_{\Phi_k},\Z)$ as a $\Z$-module  is
\begin{equation*}
    \rho_r \colonequals  [A^r\tau]-[\tau]\,\,\text{ for }\,\,1 \le r \le n - 1.
\end{equation*}
\end{proposition} 
\begin{proof}
We first show that $\rho_r\colonequals[A^r\tau]-[\tau]$ is a well-defined element in $\rH_1(X_{\Phi_k},\Z)$ for all $1\le r\le n-1$.  It suffices to show that \[\pi(A^r\tau\cdot 0) = \pi(\tau\cdot 0)\text{ and }\pi(A^r\tau\cdot i\infty) = \pi(\tau\cdot i\infty).\]
By \eqref{Etauendpoints}, the equalities are equivalent to
\[\pi(A^r \cdot 1) = \pi(1)\text{ and }\pi(A^r\cdot 0) = \pi(0).\]

Recall that $B$ fixes $0$.
So $A^{r}\cdot0 = A^rB^{-(k^{-1})\bmod n} \cdot 0$ for $r\in\{1,\ldots,n-1\}$.
By Lemma~\ref{Lcosetsphik}, $\pi(A^{r} \cdot 0) = \pi(A^{r+1} \cdot 0)$. 
Thus $\pi(A^r \cdot 0) = \pi( {\mathrm{Id}} \cdot 0) = \pi(0)$.

For the other equality, as in \cite[page 2308]{E:integral}, consider the degree $n$ cover 
$X_{\Phi_k} \to X_{\Gamma(2)}\cong \bar\frakH/\Gamma(2)\cong \mathbb{P}^1$.
By Lemma~\ref{Lindex}, this corresponds to the cover $W_k \to \mathbb{P}^1$, given by $(u,v) \mapsto u$ in affine coordinates.
The fact that this cover is totally ramified at $\eta_1$ over $u=1$ implies that 
the same is true for $X_{\Phi_k} \to \mathbb{P}^1$.
The elements $\{\pi(A^r \cdot 1)\}_{0\le r\le n-1}$ all lie above the point corresponding to $u = 1$, 
so it follows that $\pi(A^r\cdot 1) = \pi(1)$ for all $1\le r\le n-1$. 

To see that the elements in $\{\rho_r\}_{1\le r\le n-1}$ are $\Z$-linearly independent, we assume that
$\sum_{r=1}^{n-1}a_r\rho_r=0$,
with $a_r\in \Z$. This implies that
$\sum_{r=1}^{n-1}a_r[A^r\tau]-\left(\sum_{r=1}^{n-1}a_r\right)[\tau]=0$.
By Proposition~\ref{Pgenerators_modular_group}, the set $\{[A^r\tau]\}_{0\le r\le n-1}$ is linearly independent, so $a_r=0$ for $1\le r\le n-1$.
\end{proof}

\begin{corollary}\label{CBasisEi}
The set $\{[E_i] \mid 1\le i\le n-1\}$ is a basis for $\rH_1(W_k,\Z)$ as a $\Z$-module.
\end{corollary}

\begin{proof}
The action of $A$ is identified with $\epsilon_0$ and $\epsilon_0((x,y)) = (\zeta x, y)$.
Also $\epsilon((u,v)) = (u, \zeta v)$.
Since $v = xy^k$, this shows that $A$ acts like $\epsilon$ on $W_k=X_{\Phi_k}$.
Since $E_i=\epsilon^i\tau\cdot \tau^{-1}$ and $[E_0]=0$, this shows that $\rho_i = [E_i]$. 
The result is then immediate from Proposition~\ref{Pbasis}.
\end{proof}

\section{Description using invariants} \label{Sexamples}

As in Section~\ref{ScurveDef}, let $W=W_{n,k}$.  In this section, we illustrate Theorem~\ref{Tmaintheorem}.

\subsection{Some invariant elements of $\wedge^2 \rH_1(W)$} \label{Sinvariant}

By Corollary~\ref{CBasisEi}, $\{[E_i] \mid 1\le i\le n-1\}$ is a basis for $\rH_1(W,\Z)$.
For $1 \leq r \leq (n-1)/2$, in $\wedge^2 \rH_1(W)$, we define
\[T_r \colonequals \sum_{i=0}^{n-1} [E_i] \wedge [E_{i+r}].\]
Note that both $E_i$ and $T_r$ implicitly depend on $k$.
For simplicity of notation, we write $E_i$ rather than $[E_i]$ in the homology in the proofs in this section.

\begin{lemma}
The element $T_r$ is invariant under the automorphism $\epsilon$.
\end{lemma}

\begin{proof} 
By \eqref{Dloops}, $E_i = e_i ^{-1} \cdot e_0$.  
Then $\epsilon(E_i) = E_{i+1} \cdot E_1^{-1}$ in $\pi_1(W_k)$. 
So $\epsilon(E_i) = E_{i+1} - E_1$ in $\rH_1(W_k, \ZZ)$.
We compute that
\begin{eqnarray*}
\epsilon(T_r) & = & \sum_{i=0}^{n-1} \epsilon(E_i) \wedge \epsilon(E_{i+r}) 
 =  \sum_{i=0}^{n-1}(E_{i+1} - E_1) \wedge (E_{i+r + 1} - E_1) \\
& = & \sum_{i=0}^{n-1} (E_{i+1} \wedge E_{i+r+1} - E_1 \wedge E_{i+r+1} + E_1 \wedge E_{i+1}).
\end{eqnarray*} 

Then $\sum_{i=0}^{n-1}(- E_1 \wedge E_{i+r+1} + E_1 \wedge E_{i+1})=0$
and so $\epsilon(T_r) = \sum_{i=0}^{n-1} E_{i+1} \wedge E_{i+r+1} = T_r$.
\end{proof}

\subsection{Applications} \label{Sapplications}

By Theorem~\ref{Tmaintheorem}, the classifying element has the form 
$$\Delta=\sum_{1\le i< j\le n-1}c_{i,\,j}[E_i]\wedge [E_j].$$
In this section, we describe $\Delta$ using the invariant elements from Section~\ref{Sinvariant}.

\begin{corollary} \label{Capplicationk12}
Let $n, k$ be integers such that $1 \leq k \leq n-2$ and $\mathrm{gcd}(n, k(k+1))=1$.
Let $c \in \{1,...,n-1\}$ be such that $c \equiv (k+1)^{-1} \bmod n$.
Then $\Delta = \sum_{r \in S_{n,k}} (-T_{r-1} + T_r)$,
where \[S_{n,k} = \{r \in \ZZ/n\ZZ \mid r \equiv j(k+1) \bmod n \mbox{ for some } 1 \leq j \leq c-1\}.\]
\end{corollary}

\begin{proof}
This is immediate from Theorem~\ref{Tmaintheorem}.
\end{proof}

\begin{corollary} With the same hypotheses as Corollary~\ref{Capplicationk12}:
\begin{enumerate}
\item If $k=1$, then $\Delta= \sum_{r=1}^{(n-1)/2} (-1)^{r} T_r$.

\item If $k=2$, then $\Delta = \sum_{r=1}^{(n-1)/2} w_r T_r$ where $w_r = 1$ if $r \equiv 0 \bmod 3$ and 
$w_r = -1$ if $r \equiv n \bmod 3$.

\item If $k = n-2$, then $\Delta = -T_1$. 

\item If $k=n-3$, then $\Delta= \sum_{r=2}^{(n-1)/2} (-1)^{r-1} T_r$.

\item If $k=(n-1)/2$, then $\Delta = -T_{(n-1)/2}$.

\item If $k=(n-3)/2$, then $\Delta = T_{(n-1)/2}-T_1$.

\item If $n \equiv 2 \bmod 3$ and $k=(n-2)/3$, then $\Delta = T_{k+1}-T_k $.

If $n \equiv 1 \bmod 3$ and $k=(2n-2)/3$, 
then $\Delta = T_{(k+2)/2}-T_{k/2} $.
\end{enumerate}
\end{corollary}

\begin{proof}
\begin{enumerate}
\item If $k=1$, then $c=(n+1)/2$.
Then $r \in S_{n,k}$ if and only if $r$ is even.

\item Write $c=(n+1)/3$ if $n \equiv 2 \bmod 3$ and $c=(2n+1)/3$ if $n \equiv 1 \bmod 3$.
Write $\ell = n-3$.  Then 
\[c_{I,J} = \begin{cases}
-1 & \text{ if } I-J \equiv \ell+1,\ 2\ell+1,\ \ldots,\ (c-1)\ell+1 \bmod n;\\ 
+1 & \text{ if } I-J \equiv \ell,\ 2\ell,\ \ldots,\ (c-1)\ell \bmod n;\\
0 & otherwise. 
\end{cases}\]

\item If $k=n-2$, then $c=n-1$.  In this case, 
we cover every sheet $0,1,2,\ldots, n-2, n-1$ in order, starting with sheet $0$ and ending with sheet $n-1$. 
It is more convenient to consider the edges between $\tau$ and $\tau^{-1}$, which are $(\epsilon^{n-1} \tau)^{-1}$
and $\epsilon \tau$.  So $E_0 \wedge E_1$ and $-E_{n-1} \wedge E_0$ appear in $\Delta$.
So $c_{I,J} = -1$ (resp.\ $+1$) if $J-I \equiv 1 \bmod n$ (resp.\ $J-I \equiv n-1 \bmod n$) and $c_{I,J}=0$ otherwise.

\item If $k=n-3$, then $c=(n-1)/2$.  We omit the details.  
%By a similar proof, we see
%\[c_{I,J} = \begin{cases}
%-1 & \text{ if } J-I \text{ odd, not } 1 \bmod n;\\
%+1 & \text{ if } J-I \text{ even, not } n-1 \bmod n;\\
%0 & \text{ if } J-I \equiv 1, n-1 \bmod n.
%\end{cases}\]

\item If $k=(n-1)/2$, then $c=2$.  It follows that
$c_{I,J} = -1$ (resp.\ $+1$) only if $J-I  \equiv (n-1)/2 \bmod n$ 
(resp.\ $(n+1)/2  \bmod n$). 

\item If $k=(n-3)/2$, then $c=n-2$.  In this case, we cover all but the $(n-1)$st sheet, starting with sheet $0$, then 
$(n+1)/2$, then $1$, then $(n+3)/2$, etc.
It is more convenient to consider the edges between $\tau$ and $\tau^{-1}$, which are
$(\epsilon^{n-1} \tau)^{-1}$, $\epsilon \tau$, $\epsilon^{(n+1)/2} \tau$, and $(\epsilon^{(n-1)/2} \tau)^{-1}$.
So $c_{I,J} = -1$ if $J-I \equiv (n+1)/2  \text{ or } 1 \bmod n$,  $c_{I,J} = +1$ if $J-I \equiv (n-1)/2  \text{ or } n-1 \bmod n$, and 
$c_{I,J} = 0$ otherwise. 

\item If $k=(n-2)/3$ and $n \equiv 2 \bmod 3$, (resp.\ $k=(2n-2)/3$ and $n \equiv 1 \bmod 3$), 
then $c=3$.  Let $\ell = n-(k+1)$ which equals $(2n-1)/3$ if $n \equiv 2 \bmod 3$ and equals $(n-1)/3$ 
if $n \equiv 1 \bmod 3$.  Then $c_{I,J} = -1$ if $J-I \equiv \ell, 2\ell \bmod n$ and $c_{I,J} = +1$ if $J-I \equiv \ell+1, 2\ell +1 \bmod n$.
\end{enumerate}
\end{proof}

Corollary~\ref{Capplicationk12} determines $\Delta$ for all values of $k$ when $3 \leq n < 11$
and for all but two values of $k$ when $n=11$.
We include these two as final examples.

\begin{example}
If $n=11$ and $k=6$, then $c = 8$ and $\Delta = -T_1+T_3-T_4$.
\end{example}

\begin{example}
If $n=11$ and $k=7$, then $c = 7$ and $\Delta = -T_1+T_2-T_3+T_5$.
\end{example}

%SAVE \begin{proof}
%Let $k=6$.  Then $c=8$.  By Theorem~\ref{Tmaintheorem}, we obtain:
%\begin{align*}
%\Delta=&E_1\wedge \left(-E_2  +  E_4  -  E_5  +  E_8  - E_9 \right) +  E_2\wedge\left(-E_3  +  E_5  - E_6  +  E_9  - E_{10}\right)\\
%	& +  E_3\wedge \left(-E_4  +  E_6  -   E_7  +  E_{10}\right) +  E_4\wedge \left(-E_5  +   E_7  - E_8 \right) +  E_5\wedge \left(-E_6  +  E_8  -  E_9 \right)\\
%	& +  E_6\wedge \left(-E_7  +  E_9  -  E_{10} \right) +  E_7\wedge \left(-E_8  +  E_{10}  \right)\\
%	& +  E_8\wedge (-E_9) +  E_9\wedge (-E_{10})\\
%	=&-T_1+T_3-T_4.
%\end{align*}

%Let $k=7$.  Then $c=7$.  
%By Theorem~\ref{Tmaintheorem}, we obtain:
%\begin{align*}
%\Delta=&E_1\wedge \left(-E_2 + E_3 - E_4 + E_6 -  E_7 + E_9 - E_{10}\right) \\
%&+ E_2\wedge \left(-E_3 + E_4 - E_5 +  E_7 - E_8 + E_{10}\right)+ E_3\wedge\left(-E_4 + E_5 - E_6 + E_8 - E_9\right) \\
%&+ E_4\wedge \left(-E_5 +  E_6 -  E_7 + E_9 - E_{10}\right)+ E_5\wedge \left(-E_6 + E_7 -  E_8 + E_{10}\right) \\
%&+ E_6\wedge \left(-E_7 + E_8 - E_9\right) + E_7\wedge \left(-E_8 + E_9 - E_{10}\right)+ E_8\wedge \left(-E_9 + E_{10}\right) \\&+ E_9\wedge (-E_{10})\\
%=&-T_1+T_2-T_3+T_5.
%\end{align*}
%\end{proof}

\bibliographystyle{amsalpha}
\bibliography{biblio.bib}

\end{document}